\documentclass[12pt]{article}
\usepackage{amsthm}
\usepackage{amsfonts}
\usepackage{amsmath}
\usepackage{amssymb}
\usepackage{colonequals}
\usepackage{epsfig}
\usepackage{url}
\newcommand{\GG}{{\cal G}}

\newtheorem{theorem}{Theorem}
\newtheorem{conj}[theorem]{Conjecture}
\newtheorem{corollary}[theorem]{Corollary}
\newtheorem{lemma}[theorem]{Lemma}

\newtheorem{observation}[theorem]{Observation}

\newcommand{\col}{\text{col}}

\newcommand{\eps}{\varepsilon}
\newcommand{\mc}[1]{\mathcal{#1}}
\newcommand{\brm}[1]{\operatorname{#1}}
\newcommand{\bb}[1]{\mathbb{#1}}

\title{Islands in minor-closed classes. I. Bounded treewidth and separators}
\author{Zden\v{e}k Dvo\v{r}\'ak\thanks{Charles University, Prague, Czech Republic.
E-mail: {\tt rakdver@iuuk.mff.cuni.cz}.  Supported by project 17-04611S (Ramsey-like aspects of graph
coloring) of Czech Science Foundation.}\\
\and Sergey Norin\thanks{Department of Mathematics and Statistics, McGill University. Email: {\tt snorin@math.mcgill.ca}. Supported by an NSERC Discovery grant.}}
\date{}

\begin{document}
\maketitle

\begin{abstract}
The \emph{clustered chromatic number} of a graph class is the minimum integer
$t$ such that for some $C$ the vertices of every graph in the class can be
colored in $t$ colors so that every monochromatic component has size at most
$C$.  We show that the clustered chromatic number of the class of graphs
embeddable on a given surface is four, proving the conjecture of Esperet and
Ochem. Additionally, we study the list version of the concept and
characterize the minor-closed classes of graphs of
bounded treewidth with given clustered list chromatic number. We further strengthen the above results to solve some extremal problems on bootstrap percolation of  minor-closed classes.
\end{abstract}

\section{Introduction}\label{sec-intro}
Let $G$ and $H$ be graphs.  A \emph{model} of $H$ in $G$ is a function $\mu$ assigning to vertices of $H$ pairwise vertex-disjoint
non-empty connected subgraphs of $G$ such that for every $uv\in E(H)$, there exists an edge of $G$ with one end in $\mu(u)$
and the other end in $\mu(v)$.   If $H$ has a model in $G$, we say that $H$ is a \emph{minor} of $G$, otherwise, we say that $G$ is \emph{$H$-minor-free}. A class $\mc{G}$ of graphs is \emph{a (proper) minor-closed class} if $\mc{G}$ does not contain all graphs, and  for every graph $G \in \mc{G}$  every minor of $G$ also belongs to $\mc{G}$. We define the \emph{chromatic number $\chi(\mc{G})$} of a graph class $G$ as the minimum integer $t$ such that every graph $G \in \mc{G}$ is properly $t$-colorable
(and $\chi(\mc{G})=\infty$ if no such $t$ exists). 
The famous Hadwiger's conjecture can be considered as a characterization of minor-closed graph classes with given chromatic number.

\begin{conj}[Hadwiger's conjecture~\cite{Hadwiger}] 
Let $t \geq 1$ be an integer, and let $\mc{G}$ be a minor-closed class of graphs. Then $\chi(\mc{G}) \leq t$ if and only if $K_{t+1} \not \in \mc{G}$.
\end{conj}

For integers $C\ge 1$ and $t\ge 0$, a \emph{$t$-coloring with clustering $C$} of a graph $G$ is a (not necessarily proper) coloring of vertices of $G$ using $t$ colors
such that $G$ contains no monochromatic connected subgraph with more than $C$ vertices. We denote by $\chi_C(G)$ the minimum $t$ such that $G$ admits a $t$-coloring with clustering $C$. The \emph{clustered chromatic number $\chi_\star(\mc{G})$} of a graph class $\mc{G}$ is the minimum $t$ such that there exists $C$ so that $\chi_C(G) \leq t$ for every graph $G \in \mc{G}$
($\chi_\star(\mc{G})=\infty$ if no such $t$ exists).
The clustered chromatic number of minor-closed classes of graphs has been recently actively investigated, motivated in part by Hadwiger's conjecture. 
Let $\mc{X}(H)$ denote the minor-closed class consisting of all $H$-minor-free graphs.  Kawarabayashi and Mohar~\cite{KawMoh07} proved the first linear bound on $\chi_\star(\mc{X}(K_{t}))$ showing that $\chi_\star(\mc{X}(K_{t})) \leq \lceil \frac{31t}{2} \rceil$.\footnote{No such linear bound is known for $\chi(\mc{X}(K_{t}))$.} The upper bound on $\chi_\star(\mc{X}(K_{t}))$ has been successively improved in~\cite{Wood10,EKKOS15,LiuOum15}. Most recently, using a beautiful self-contained argument, van den Heuvel and Wood~\cite{vdHW17} have shown that $\chi_\star(\mc{X}(K_{t})) \leq 2t-2$.  

The above definitions can be straightforwardly extended from coloring to list coloring, and many of our techniques extend as well. A \emph{$t$-list assignment $L$} for a graph $G$ assigns a finite set $L(v)$ of size at least $t$ to every vertex $v \in V(G)$. An \emph{$L$-coloring  with clustering $C$} of a graph $G$ assigns to every vertex $v$ a color from $L(v)$ such that $G$ contains no monochromatic connected subgraph with more than $C$ vertices. Let $\chi^l_C(G)$ denote the minimum $t$ such that $G$ admits an $L$-coloring  with clustering $C$ for every $t$-list assignment $L$.
Clearly, $\chi^l_C(G) \geq \chi_C(G)$. The \emph{clustered list chromatic number $\chi^l_\star(\mc{G})$} of a graph class $\mc{G}$ is the minimum $t$ such
that there exists $C$ so that $\chi^l_C(G) \leq t$ for every graph $G \in \mc{G}$ (and $\chi^l_\star(\mc{G})=\infty$ if no such $t$ exists).

In the sequel to this paper 
we show that $\chi_\star(\mc{X}(K_{t})) = \chi^l_\star(\mc{X}(K_{t})) = t-1$, proving the weakening of Hadwiger's conjecture for clustered chromatic number.\footnote{Note that Hadwiger's conjecture is false for list colorings.} In the current paper we introduce two less technical ingredients of the proof, which might be of independent interest. In particular, our results imply the above mentioned bound on $\chi_\star(\mc{X}(K_{t}))$ of van den Heuvel and Wood, using very different techniques.

Rather than working with the clustered (list) chromatic number, we bound a different parameter which dominates it.
The \emph{coloring number} $\col(G)$ of a graph $G$ is the smallest integer $t$ such that every non-empty subgraph of $G$ contains a vertex of degree less than $t$.
A standard greedy argument shows that $\chi(G)\le \col(G)$.  Let us now introduce a similar notion for clustered coloring, first defined by Esperet and Ochem~\cite{EspOch16}. 
A \emph{$t$-island} in a graph $G$ is a non-empty subset $S$ of vertices of $G$ such that each vertex of $S$ has less than $t$ neighbors in $V(G)\setminus S$.
Let $\col_C(G)$ denote the smallest integer $t$ such that each non-empty subgraph of $G$ contains a $t$-island of size at most $C$.
Hence, $\col_1(G)=\col(G)$.  The following analogue of the bound $\chi(G)\le \col(G)$  holds for the coloring with bounded clustering.

\begin{lemma}\label{lemma-coloring}
For every integer $C\ge 1$, every graph $G$ satisfies $\chi^l_C(G)\le \col_C(G)$.
\end{lemma}
\begin{proof} Let $L$ be a $t$-list assignment for $G$.
Suppose that each non-empty subgraph of $G$ contains a $t$-island of size at most $C$.  By induction on the size of $G$, we show that $G$ has an
$L$-coloring with clustering at most $C$.  This is trivial if $G$ is empty, hence assume that $V(G)\neq\emptyset$. Let $S$ be a $t$-island in $G$ of size at most $C$.
By the induction hypothesis, $G - S$ has an $L$-coloring with component size at most $C$.  Color each vertex of $v \in S$ by an arbitrary color from $L(v)$ different from
the colors of its neighbors in $V(G)\setminus S$.  Clearly, each connected monochromatic subgraph of $G$ is contained either in $S$ or in $V(G)\setminus S$, and since $|S|\le C$,
we conclude that we obtained an $L$-coloring of $G$ with clustering at most $C$.
\end{proof}

For a class $\GG$ of graphs, let the \emph{clustered coloring number $\col_\star(\GG)$ of $\mc{G}$} denote the smallest integer $t$ such that there exists $C\ge 1$ such that
$\col_C(G)\le t$ for every graph $G\in\GG$ ($\col_\star(\GG)=\infty$ if no such $t$ exists). By Lemma~\ref{lemma-coloring}, $\col_\star(\mc{G})  \geq \chi^l_\star(\mc{G}) \geq \chi_\star(\mc{G})$, i.e. any upper bound on $\col_\star(\mc{G})$ gives an upper bound on $\chi_\star(\mc{G})$. This observation motivates our investigation of the clustered coloring number in this paper.

It follows from Observation~\ref{obs-basgr} below that 
$\col_\star(\mc{X}(K_{t+1})) \geq t$.  On the other hand, for every $n\ge 1$,
there exist minor-closed graph classes $\GG$ containing $K_n$ with $\col_\star(\GG)=1$, e.g., the class of all graphs with at most $n$
vertices.  Thus the complete minors present in the minor-closed class do not determine the clustered coloring number of the class. This motivates us to ask for an exact description of minor-closed classes $\GG$ with $\col_\star(\GG)\le t$.

Two kinds of graphs seem to play an important role in this context.  One of them are the complete bipartite graphs $K_{n,m}$.
For the other one, let $I_n+P_m$ denote the graph consisting of a path on $m$ vertices and $n$ additional vertices adjacent to all the vertices of the path.
The following is straightforward, and also follows from a stronger Lemma~\ref{lem-listksm} below.

\begin{observation}\label{obs-basgr}
Let $t\ge 1$ be an integer.  For every $C\ge 1$, there exists $m\ge 1$ such that all $t$-islands in $K_{t,m}$ and $I_{t-1}+P_m$
have more than $C$ vertices.
\end{observation}

Thus a class $\GG$ with $\col_\star(\GG)\le t$ can contain only finitely many graphs of form $K_{t,m}$ or $I_{t-1}+P_m$.  We conjecture
that for minor-closed classes, this condition is also sufficient.

\begin{conj}\label{conj-chara}
A minor-closed class of graphs $\GG$ satisfies $\col_\star(\GG)\le t$ if and only if there exists $m\ge 1$ such that
$K_{t,m}\not\in \GG$ and $I_{t-1}+P_m\not\in\GG$.
\end{conj}

Conjecture~\ref{conj-chara} implies that $\col_\star(\mc{X}(K_{t+1}))=t$ (and hence $\chi_\star(\mc{X}(K_{t+1})) = t$), since $K_{t+1}$ 
is a minor of $K_{t,m}$ and $I_{t-1}+P_m$ for all $m\ge t$. In the next section we show that if Conjecture~\ref{conj-chara} holds, then $\chi^l_\star(\GG)=\col_\star(\GG)$
for every minor-closed graph class $\GG$. On the other hand, the parameters $\chi_\star$ and $\chi^l_\star$ are not tied, and thus, unfortunately, one can not hope to extend
the methods of this paper to characterize minor-closed graph classes with given clustered chromatic number. We refer the reader to~\cite[Conjectures 30 and 32]{NSSW17} for a conjectured characterization.

In the next section we state our main results and present several applications, including proofs of two conjectures of Esperet and Ochem.
We prove our two main results in Sections~\ref{sec-percolation} and~\ref{sec-treedec}.

\section{Our results}

We start this section by presenting a strengthening of Observation~\ref{obs-basgr} to clustered list chromatic number.

\begin{lemma}\label{lem-listksm} 
For all positive integers $C,t$ there exists a positive integer $m$ such that $\chi_C^l(K_{t,m}) \geq t+1$ and $\chi_C^l(I_{t-1}+P_m) \geq t+1$.
\end{lemma}

\begin{proof} We present the proof for $G=I_{t-1}+P_m$, the proof for $K_{t,m}$ is similar.
Since $\chi_C(G)\ge 2$ for every connected graph $G$ with more than $C$ vertices, we can assume that $t\ge 2$.
Let $S$ be a set of size $(t-1)t$.
Let $L$ be a $t$-list assignment for $G$ such that $L(v) \subseteq S$ for all $v \in V(G)$ and the vertices of $I_{t-1}$
are assigned disjoint subsets of $S$. Suppose further that for every subset $T \subseteq S$ with $|T|=t$ there exists a set $Q_T$ of
$tC^2$ consecutive vertices of $P_m$ such that $L(v)=T$ for every $v \in Q_T$. Clearly such a list assignment exists if $m$ is sufficiently large.

Consider an $L$-coloring of $G$. It remains to show that there exists a monochromatic connected subgraph of size more than $C$. Let $T'$
be the set of colors assigned to vertices of $I_{t-1}$. Then $|T'|=t-1$ by the choice of $L$. Let $T=T' \cup \{c\}$ for some color $c \in S \setminus T'$,
and let $Q_T$ be as defined in the previous paragraph. If some color in $T'$ is used on at least $C$ vertices of $Q_T$ then $G$ contains
a monochromatic connected subgraph induced by these vertices and a vertex of $I_{t-1}$. Otherwise, $G$ contains a monochromatic subpath
of $P_m$ of length at least $(C+1)$ induced by the vertices in $Q_T$ which are colored using color $c$.
\end{proof}

Lemma~\ref{lem-listksm} immediately implies the following.

\begin{corollary}\label{cor-liststar}
Let $\GG$ be a class of graphs such that $\chi^l_\star(\GG)\le t$. Then there exists $m\ge 1$ such that
$K_{t,m}\not\in \GG$ and $I_{t-1}+P_m\not\in\GG$.
\end{corollary}

Corollary~\ref{cor-liststar} in particular implies that if Conjecture~\ref{conj-chara} holds then $\chi^l_\star(\GG)=\col_\star(\GG)$ for every minor-closed graph class $\GG$, as mentioned in the introduction.

\vskip 10pt

Let $\brm{tw}(G)$ denote the treewidth of the graph $G$.\footnote{Treewidth of a graph is defined in Section~\ref{sec-treedec}.}
We say that a class of graphs $\GG$ is \emph{of bounded treewidth} if there exists an integer $w$ such that $\brm{tw}(G) \leq w$ for every graph $G \in \GG$.
Our first main result  proves Conjecture~\ref{conj-chara} in the special case of minor-closed classes of graphs of bounded tree-width
(or equivalently according to a result of Robertson and Seymour~\cite{RSey}, minor-closed classes that do not contain all planar graphs).

\begin{theorem}\label{thm-bndtw}
Let $\GG$ be a minor-closed class of graphs of bounded treewidth, and let $t\ge 1$ be an integer.  Then $\col_\star(\GG)\le t$ if and only if there exists $m\ge 1$ such that
$K_{t,m}\not\in \GG$ and $I_{t-1}+P_m\not\in\GG$.
\end{theorem}

Theorem~\ref{thm-bndtw} and Corollary~\ref{cor-liststar} imply the following.

\begin{corollary}\label{cor-liststar2}
If $\GG$ is a minor-closed class of graphs of bounded treewidth, then $\chi^l_\star(\GG) = \col_\star(\GG)$. 
\end{corollary}

Theorem~\ref{thm-bndtw} can be applied to bound the clustered chromatic number of minor-closed classes of unbounded treewidth. The key tool which allows such an application is the following theorem of DeVos et al.

\begin{theorem}[DeVos  et al.~\cite{DDOSRSV}]\label{thm-devos} For every minor-closed class $\GG$ there exists an integer $w$, 
such that for every graph $G \in \GG$ there exists a partition $V_1,V_2$ of $V(G)$ satisfying $\brm{tw}(G[V_1]) \leq w$
and $\brm{tw}(G[V_2]) \leq w$.
\end{theorem}

For ordinary (not list) clustered coloring, one can use disjoint sets of colors on parts $V_1$ and $V_2$ of
such a partition.
Hence, combining Theorems~\ref{thm-bndtw} and \ref{thm-devos} and Lemma~\ref{lemma-coloring} yields the following.

\begin{corollary}\label{cor-double} 
Let $\GG$ be a minor-closed class of graphs, and let $t,m\ge 1$ be integers such that
$K_{t,m}\not\in \GG$ and $I_{t-1}+P_m\not\in\GG$. Then $\chi_\star(\mc{G}) \leq 2t$. 

In particular, \begin{equation}\label{ineq-kst}
\chi_\star(\mc{X}(K_{t,m})) \leq 2t+2
\end{equation} and 
\begin{equation}\label{ineq-kt}
\chi_\star(\mc{X}(K_{t+1})) \leq 2t.
\end{equation}
\end{corollary}

As mentioned in the introduction, a different proof of the bound (\ref{ineq-kt}) is given by van den Heuvel and Wood~\cite{vdHW17}, while (\ref{ineq-kst}) improves on the bound $\chi_\star(\mc{X}(K_{t,m})) \leq 3t$ established in~\cite{vdHW17}.

\vskip 10pt

Our second main result bounds the clustered coloring number of minor-closed classes of graphs in terms of the maximum density of the class.

\begin{theorem}\label{thm-sparse}
For every graph $H$, integer $t \geq 1$  and  real $\alpha > 0$ there exists $C > 0$ satisfying the following. Let $G$ be an $H$-minor-free graph such that $|E(G)|< (t-\alpha)|V(G)|$. Then $G$ contains a $t$-island of size at most $C$.  
\end{theorem} 

Theorem~\ref{thm-sparse} immediately implies the following.

\begin{corollary}\label{cor-sparse}
Let $\GG$ be a class of graphs closed under taking subgraphs, such that $\GG \subseteq \mc{X}(H)$ for some graph $H$. Let real $d > 0$ be such that $|E(G)|\leq d|V(G)|+o(|V(G)|)$ for all graphs $G\in\GG$. Then $\col_\star(\mc{G}) \leq  \lfloor d+1 \rfloor$.
\end{corollary}

For a surface $\Sigma$ let $\mc{G}(\Sigma)$ denote the (minor-closed) class of graphs embeddable on $\Sigma$. 
Kleinberg, Motwani,
Raghavan, and Venkatasubramanian~\cite{KMRV97} and, independently, Alon, Ding, Oporowski and Vertigan~\cite{ADOV03} proved that $\chi_\star(\mc{G}(\Sigma)) \geq 4$ for every surface $\Sigma$. Thus $\col_\star(\mc{G}(\Sigma)) \geq 4$. By Euler's formula, for every surface $\Sigma$ there exists a constant $c(\Sigma)$ such that $|E(G)| \leq 3|V(G)|+c(\Sigma)$ for all graphs $G\in\GG(\Sigma)$. Thus Corollary~\ref{cor-sparse} implies that the above lower bounds on 
  $\col_\star(\mc{G}(\Sigma))$ and $\chi_\star(\mc{G}(\Sigma)) \geq 4$ are  tight, as conjectured by Esperet and Ochem~\cite[Conjecture 3]{EspOch16}.
  
\begin{corollary}\label{cor-surface} For every surface $\Sigma$
$$\chi_\star(\mc{G}(\Sigma)) = \chi^l_\star(\mc{G}(\Sigma)) =\col_\star(\mc{G}(\Sigma)) = 4.$$ 
\end{corollary}  

Let $\mc{G}_5(\Sigma)$ be the class of graphs with girth at least five which are embeddable on $\Sigma$. The question of determining  $\chi_\star(\mc{G}_5(\mathbb{R}^2))$ was considered, in particular, in~\cite{AUW17}, and was until now open.
 Euler's formula again gives us a density bound $|E(G)| \leq 5/3|V(G)|+c(\Sigma)$ for some $c(\Sigma)$ and all $G\in\mc{G}_5(\Sigma)$, and thus Corollary~\ref{cor-sparse} implies the following bound on $\col_\star(\mc{G}(\Sigma))$ conjectured in~\cite[Conjecture 7]{EspOch16}. 
 
\begin{corollary}\label{cor-girth} For every surface $\Sigma$
$$\chi_\star(\mc{G}_5(\Sigma)) = \chi^l_\star(\mc{G}_5(\Sigma)) = \col_\star(\mc{G}_5(\Sigma)) = 2.$$ 
\end{corollary}  

Corollary~\ref{cor-sparse} can also be used to determine $\col_\star(\mc{X}(K_t))$ for $t \leq 9$. By the results of ~\cite{Dirac64,Jorgensen94,Mader68,SonTho06} if $G$ is a $K_t$-minor free graph for some $t \leq 9$ then $|E(G)| \leq (t-2)|V(G)|$. This implies the following.

\begin{corollary}\label{cor-smallkt} Let $1 \leq t \leq 9$ be an integer. Then $$\chi_\star(\mc{X}(K_t)) = \chi^l_\star(\mc{X}(K_t)) =\col_\star(\mc{X}(K_t)) = t-1.$$ 
\end{corollary}

Finally, let us discuss a relationship to another concept, \emph{bootstrap percolation}.  Consider the following process on a graph $G$ for some integer $t\ge 0$.  Let vertices of some set $A_0\subseteq V(G)$
be marked active.  If there exists an inactive vertex $v$ in $G$ with at least $t$ active neighbors, $v$ becomes active.  We repeat this procedure until there are no more inactive vertices
with at least $t$ active neighbors.  If at the end, all vertices of $G$ are active, we say that the set $A_0$ \emph{$t$-percolates}. Bootstrap percolation was introduced by Chalupa, Leath and Reich~\cite{CLR79} as a simplification of existing models of ferromagnetism. Extremal problems for bootstrap percolation, similar to the ones we consider in this paper, were studied for very structured graph families e.g. in~\cite{BalBol06,MorNoe15}. 

For $\varepsilon>0$, we say that a graph $G$ is \emph{$\varepsilon$-resistant to $t$-percolation} if no set of at most $\varepsilon|V(G)|$ vertices of $G$ $t$-percolates.
For a class of graphs $\GG$, let us define the \emph{percolation threshold $p(\GG)$}
of $\GG$ to be the minimum integer $t$ such that for some $\varepsilon>0$, all non-null graphs in $\GG$
are $\varepsilon$-resistant to $t$-percolation ($p(\GG)=\infty$ if no such $t$ exists).  Clearly, all graphs in $\GG$ are also $\varepsilon$-resistant to $t'$-percolation for every $t'\ge t$.

How to show that a class is percolation resistant?  Observe that a set $t$-percolates if and only if no $t$-island is contained in its complement.  If a graph $G$ contains more than $\varepsilon |V(G)|$
pairwise disjoint $t$-islands, then each set of size at most $\varepsilon|V(G)|$ is disjoint from at least one of them, and thus it does not $t$-percolate.  Hence, the following notion gives an upper bound to
the percolation threshold.  Let the \emph{pervasive clustered coloring number $\text{pcol}_\star(\GG)$} denote the minimum integer $t$ such that for some $\varepsilon>0$, each graph $G\in\GG$ contains at least $\varepsilon |V(G)|$ pairwise disjoint $t$-islands
($\text{pcol}_\star(\GG)=\infty$ if no such $t$ exists).

If $G$ contains linearly many pairwise disjoint $t$-islands, some of them must have constant size.  Hence, for a subgraph-closed class $\GG$ we have inequalities $\text{pcol}_\star(\GG)\ge p(\GG)$ and $\text{pcol}_\star(\GG)\ge \col_\star(\GG)$.
In general, percolation threshold may be smaller than the clustered coloring number; e.g., for $t\ge 3$ the class of graphs $G$ with maximum degree at most $t$ and girth $\Omega(\log(|V(G)|))$
has clustered coloring number $t$ (since in $t$-regular graphs in the class, all $(t-1)$-islands must contain cycles) but all the graphs in this class are $(t-2)/(2t-2)$-resistant to $(t-1)$-percolation
(since the complements of sets of size at most $\tfrac{t-2}{2t-2}|V(G)|$ contain a vertex of degree at most $t-2$, or two vertices of degree $t-1$ joined by a path, or a cycle, forming a $(t-1)$-island).

However, in Section~\ref{sec-percolation} we prove the following.
 
\begin{theorem}\label{thm-percol1}
	Let $\GG$ be a class of graphs closed under taking subgraphs, such that  $\GG \subseteq \mc{X}(H)$ for some graph $H$.  Then $\text{pcol}_\star(\GG)= p(\GG)$.
\end{theorem}
 
Hence for the minor closed classes percolation threshold bounds the clustered coloring number. Rather than proving Theorems~\ref{thm-bndtw} and~\ref{thm-sparse} directly, we bound the percolation threshold and the pervasive clustered coloring number of the corresponding graph classes, proving the following strengthening of Theorems~\ref{thm-bndtw} and~\ref{thm-sparse}, respectively. 

\begin{theorem}\label{thm-pbndtw}
Let $\GG$ be a minor-closed class.  If $\GG$ has bounded treewidth, then $\text{pcol}_\star(\GG)=p(\GG)=\col_\star(\GG)$ is equal to the smallest integer $t$ such that for some $m\ge 0$ neither $K_{t,m}$ nor $I_{t-1}+P_m$ belongs to $\GG$. 
\end{theorem}

\begin{theorem}\label{thm-psparse}
	For every graph $H$, integer $t>0$ and $\alpha >0$ there exists $\delta>0$ satisfying the following. Let $G$ be an $H$-minor-free graph satisfying $|E(G)| \leq (t -\alpha)|V(G)|$. Then $G$ contains at least $\delta|V(G)|$ disjoint $t$-islands. 
\end{theorem}

Note that Theorem~\ref{thm-psparse} implies the following strengthening of Corollary~\ref{cor-sparse}.

\begin{corollary}\label{cor-psparse}
Let $\GG$ be a class of graphs closed under taking subgraphs, such that $\GG \subseteq \mc{X}(H)$ for some graph $H$. Let real $d > 0$ be such that $|E(G)|\leq d|V(G)|+o(|V(G)|)$ for all graphs $G\in\GG$. Then $\text{pcol}_\star(\mc{G}) \leq  \lfloor d+1 \rfloor$.
\end{corollary}

We prove Theorems~\ref{thm-percol1} and~\ref{thm-psparse} (and thus Theorem~\ref{thm-sparse})   in Section~\ref{sec-percolation}, and we prove
Theorem~\ref{thm-pbndtw} (and thus Theorem~\ref{thm-bndtw}) in Section~~\ref{sec-treedec}.
  
\section{Percolation and clustered coloring in classes with sublinear separators}~\label{sec-percolation}

In this section we prove Theorems~\ref{thm-percol1} and~\ref{thm-psparse}. In
fact we show that these results hold not just for minor-closed classes, but for
a wider family of graph classes which admit ``good'' separators. We start by
defining this family. 

A \emph{separation} of a graph $G$ is a pair $(L,R)$ of its subgraphs such that $G=L\cup R$;
note that $V(L\cap R)$ is a vertex-cut in $G$ separating $V(L)\setminus V(R)$ from $V(R)\setminus V(L)$,
if these two sets are non-empty.
Let $f:\mathbb{N}\to\mathbb{N}$ be a non-decreasing function.  We say that a class of graphs $\mc{G}$ has \emph{$f$-separators},
if for every graph $G \in \mc{G}$ there exists a separation $(L,R)$ of $G$ of order at most $f(|V(G)|)$ such that
$|V(L)\setminus V(R)|,|V(R)\setminus V(L)|\le \frac{2}{3}|V(H)|$.
We say that a function $f$ is \emph{significantly sublinear} if the sum $$\sum_{i\ge 0} \frac{f((3/2)^i)}{(3/2)^i}$$ is finite\footnote{Note that this is the case for example if $f(n)=n/g(n)$ for some function $g(n)=\Omega(\log^{1+\varepsilon} n)$ with $\varepsilon>0$.}.
We use an argument of Lipton and Tarjan~\cite{LipTar80} to prove the following.

\begin{lemma}\label{lemma-bdedcomp}
	Let $f:\mathbb{N}\to\mathbb{N}$ be a non-decreasing significantly sublinear function.
	Let $\mc{G}$ be a class of graphs closed under taking subgraphs that has $f$-separators.
	For every $\varepsilon>0$ there exists $C$ as follows. For every $n$-vertex graph $G \in \mc{G}$ there exists $X \subseteq V(G)$ such that
	$|X|\le \varepsilon n$ and every component of $G - X$ has at most $C$ vertices.  
\end{lemma}
\begin{proof}
	Since $f$ is significantly sublinear, there exists $i_0$ such that
	$$\sum_{i\ge i_0} \frac{f((3/2)^i)}{(3/2)^i}\le \frac{2}{3}\varepsilon.$$
	Let $C=\lceil (3/2)^{i_0}\rceil$.
	
	Without loss of generality, we can assume that $G$ is connected, as otherwise we can find the set $X$ separately in each component.
	Let us define a rooted tree $T$ and a mapping $\theta$ of its vertices to connected induced subgraphs of $G$ as follows.
	For the root $r$ of $T$, we set $\theta(r)=G$.  For any $v\in V(G)$, if $|V(\theta(v))|\le C$, then $v$ is a leaf of $T$.
	Otherwise, let $(L_v,R_v)$ be a separation of $\theta(v)$ order at most $f(|V(\theta(v))|)$ and let $X_v=L_v\cap R_v$.
	If $H_1$, \ldots, $H_k$ are the components of $\theta(v)-X_v$, then $v$ has $k$ children $v_1$, \ldots, $v_k$ in $T$ with
	$\theta(v_i)=H_i$ for $i=1,\ldots, k$.
	
	We let $X$ be the union of the sets $X_v$ over all non-leaf vertices $v$ of $T$.  By the construction of $T$, every component
	of $G-X$ has at most $C$ vertices, and thus it suffices to bound the size of $X$.  For $v\in V(T)$, define
	$\mathtt{rank}(v)=\lfloor \log_{3/2} |V(\theta(v))|\rfloor$.  Observe that the rank is decreasing on each path in $T$
	starting in the root, and in particular, if $\mathtt{rank}(u)=\mathtt{rank}(v)$ for distinct $u,v\in V(T)$, then
	$\theta(u)$ and $\theta(v)$ are vertex-disjoint.  For any $i\ge \lfloor \log_{3/2} C\rfloor$, let $V_i$
	be the set of non-leaf vertices $v$ of $T$ of rank $i$ and let $X_i=\bigcup_{v\in V_i} X_v$.
	Since $(3/2)^i\le |V(\theta(v))|<(3/2)^{i+1}$ for all $v\in V_i$, we have $|V_i|\le n/(3/2)^i$ and $|X_v|\le f((3/2)^{i+1})$,
	and thus $|X_i|\le \frac{f((3/2)^{i+1})}{(3/2)^i}n$.  Consequently,
	$$|X|\le \sum_{i\ge \lfloor \log_{3/2} C\rfloor} |X_i|\le n\cdot\frac{3}{2}\sum_{i\ge \lfloor \log_{3/2} C\rfloor+1}\frac{f((3/2)^i)}{(3/2)^i}\le \varepsilon n,$$
	as required.
\end{proof}

In order to prove a strengthening of Theorem~\ref{thm-psparse}, we need one more definition.
For a subset $A\subseteq V(G)$, let $e_G(A)$ (or $e(A)$ for ease of notation when the graph $G$ is clear from the context)
denote the number of edges of $G$ with at least one end in $A$. 
We say that $A\subseteq V(G)$ is a \emph{$t$-enclave} if $e_G(A) < t|A|$.
The following observation is easy, but useful.

\begin{lemma}\label{lemma-enclave}
	Let $G$ be a graph, and let $A$ be a $t$-enclave in $G$. Then there exists a $t$-island $S \subseteq A$ in $G$.
\end{lemma}

\begin{proof}
	Choose a minimal $t$-enclave $S \subseteq A$. Note that $S\neq\emptyset$, since $0\le e(S)<t|S|$.
	We claim that $S$ is an island, as desired. Suppose for a contradiction that there exists $v \in S$ with at least $t$ neighbors in $V(G)-S$. Then 
	$$e(S \setminus v) \leq e(S)-t < t|S|-t = t|S \setminus v|.$$ 
	Thus $S \setminus v$ is a $t$-enclave, contradicting the choice of $S$.
\end{proof}

Alon, Seymour and Thomas~\cite{alon1990separator} proved that any minor-closed class of graphs $\mc{G}$ has $f$-separators
for $f(n)=O(n^{1/2})$, which is significantly sublinear.
Therefore the next result implies Theorem~\ref{thm-psparse}.

\begin{theorem}\label{thm-percolgen2}
	Let $f:\mathbb{N}\to\mathbb{N}$ be a non-decreasing significantly sublinear function.
	Let $\mc{G}$ be a class of graphs closed under taking subgraphs that has $f$-separators.
	For every integer $t>0$ and $\alpha >0$ there exists $\delta>0$ so that every $G \in \mc{G}$ satisfying $|E(G)| \leq (t -\alpha)|V(G)|$ contains at least $\delta|V(G)|$ disjoint $t$-islands. 
\end{theorem}

\begin{proof}
	Let $\eps = \frac{\alpha}{2t}$.   Let $C$ be chosen to satisfy the conclusion of Lemma~\ref{lemma-bdedcomp} for this $\varepsilon$ and $\mc{G}$, and let $\delta=\frac{\alpha}{2tC}$.
	Consider an $n$-vertex graph $G \in \mc{G}$ satisfying $|E(G)| \leq (t-\alpha)n$. By the choice of $C$ there exists a set $X \subseteq V(G)$ of size
	at most $\varepsilon n$ such that every component of $G - X$ has at most $C$ vertices. Let $\mc{K}$ be the collection of vertex sets
	of the components of $G - X$, let $\mc{K}'$ be the collection of the sets in $\mc{K}$ which are $t$-enclaves,
	and let $\mc{K}'' = \mc{K} - \mc{K}'$.  By Lemma~\ref{lemma-enclave} it suffices to show that $|\mc{K}'| \geq \delta n$.
	If not, then $$\sum_{K \in \mc{K}''}|K| = n - |X| - \sum_{K \in \mc{K}'}|K| > (1 - \eps - \delta C)n = \left(1 - \frac{\alpha}{t}\right)n.$$
	Moreover, we have $e(K) \geq t |K|$ for every $K \in \mc{K}''$, and thus
	\begin{align*}
	|E(G)| \geq \sum_{K \in \mc{K''}}e(K) \geq t\sum_{K \in \mc{K''}}|K| >  t\left(1 - \frac{\alpha}{t}\right)n = (t -\alpha)n,
	\end{align*}
	which contradicts the assumption of the theorem.
\end{proof}

Let us remark that the argument of Theorem~\ref{thm-percolgen2} also directly implies Theorem~\ref{thm-sparse},
since all the $t$-enclaves we obtain have bounded size (it is possible to simplify the proof a bit in this weakened case,
since we only need to prove the existence of one such $t$-enclave).

The following result with an analogous proof implies Theorem~\ref{thm-percol1}.

\begin{theorem}\label{thm-percolgen}
	Let $f:\mathbb{N}\to\mathbb{N}$ be a non-decreasing significantly sublinear function.
	Let $\mc{G}$ be a class of graphs closed under taking subgraphs that has $f$-separators.
	Then $p(\mc{G}) = \text{pcol}_\star(\mc{G})$.
\end{theorem}

\begin{proof} Since $\text{pcol}_\star(\mc{G})\ge p(\mc{G})$ in general, it suffices to show
that $\text{pcol}_\star(\mc{G})\le p(\mc{G})$. Letting $t=p(\mc{G})$, there exists $\eps> 0$ such that $\mc{G}$ is $2\eps$-resistant to $t$-percolation.
Let $C$ be chosen to satisfy the conclusion of Lemma~\ref{lemma-bdedcomp} for 
$\varepsilon$ and $\mc{G}$.   We claim that every graph $G \in \mc{G}$ contains at least $\frac{\eps}{C}|V(G)|$ disjoint $t$-islands.
The claim implies that $\text{pcol}_\star(\mc{G}) \leq t$, and hence the theorem.
	
It remains to establish the claim. Let $G \in \mc{G}$ be an $n$-vertex graph.  By the choice of $C$, there exists $X \subseteq V(G)$ such that
$|X|\le \varepsilon n$ and every component of $G - X$ has at most $C$ vertices.
Let $\mc{K}$ be the set of all components $K$ of $G - X$ such that there exists a $t$-island $S \subseteq V(K)$ in $G$.
It suffices to show that $|\mc{K}| \geq \frac{\eps}{C}n$. If not, then the set $Z = X \cup \bigcup_{K \in \mc{K}}V(K)$
has size at most $\eps n + C|\mc{K}|\leq 2\eps n$. By the choice of $\eps$ there exists a $t$-island $S \subseteq V(G) - V(Z)$. Choose such $S$ to be minimal, then $G[S]$ is connected, and so $S \subseteq V(K)$ for some component $K$ of $G-X$. However,  $K \not \in \mc{K}$ by the choice of $Z$. This contradiction finishes the proof. 
\end{proof}	

\section{Clustered chromatic number of classes of bounded treewidth}~\label{sec-treedec}

We start by introducing the necessary concepts.
For sets $A,B$ of vertices of a graph of the same size $k$, an \emph{$A-B$ linkage} is a set of $k$ pairwise vertex-disjoint paths
with one end in $A$ and the other end in $B$.

Let $G$ and $H$ be graphs.  An \emph{$H$-decomposition} of $G$ is a function $\beta$ that to each vertex $z$ of $H$ assigns a subset of
vertices of $G$ (called the \emph{bag} of $z$), such that for every $uv\in E(G)$, there exists $z\in V(H)$ with $\{u,v\}\subseteq \beta(z)$, and
for every $v\in V(G)$, the set $\{z:v\in\beta(z)\}$ induces a non-empty connected subgraph of $H$.  When $H$ is a path or
a tree, we say that $(H,\beta)$ is a \emph{path or tree decomposition} of $G$, respectively.  The \emph{width} of an $H$-decomposition is defined as  $\max\{|\beta(x)|:x\in V(H)\}-1$, and the \emph{adhesion}
of an $H$-decomposition is  $\max\{|\beta(x)\cap\beta(y)|:xy\in E(H)\}$. The \emph{treewidth $\brm{tw}(G)$} of a graph $G$ is the minimum width of a tree decomposition of $G$.

A path decomposition $(H,\beta)$ is \emph{linked}
if for every $z\in V(H)$ with two neighbors $x$ and $y$, there exists a $(\beta(z)\cap\beta(x))-(\beta(z)\cap\beta(y))$ linkage
in $G[\beta(z)]$ (and in particular, all the intersections of adjacent bags have the same size).  Conversely, if $|\beta(z)\cap\beta(x)|=|\beta(z)\cap\beta(y)|=p$
and there exists a separation $(A,B)$ of $G[\beta(z)]$ of order less than $p$ with $\beta(z)\cap\beta(x)\subseteq V(A)$
and $\beta(z)\cap\beta(y)\subseteq V(B)$, we say that $z$ has \emph{broken bag}; by Menger's theorem, a path decomposition such that all the intersections of adjacent bags have the same size
is linked if and only if it does not contain a broken bag.  An $H$-decomposition is \emph{proper}
if $\beta(x)\not\subseteq\beta(y)$ for every $xy\in E(H)$.  The \emph{order} of an $H$-decomposition is $|V(H)|$.
For a path decomposition $(H,\beta)$, a path decomposition $(H',\beta')$ is a \emph{coarsening} of $(H,\beta)$
if there exists a model $\mu$ of $H'$ in $H$ such that $\bigcup_{z\in V(H')} V(\mu(z))=V(H)$ and $\beta'(z)=\bigcup_{x\in V(\mu(z))} \beta(x)$ for all $z\in V(H')$.
A path decomposition of $G$ is \emph{appearance-universal} if every vertex $v\in V(G)$ either appears in all bags of the decomposition,
or in at most two (consecutive) bags.  A vertex $v$ is \emph{internal} if it appears in only one bag.  A path decomposition
$(H,\beta)$ \emph{has large interiors} if for every $z\in V(H)$ with two neighbors $x$ and $y$, there exists an internal vertex
contained in $\beta(z)$
and no internal vertex has a neighbor both in $\beta(x)\setminus \beta(y)$ and in $\beta(y)\setminus\beta(x)$.


We now show how to transform a tree decomposition of large order and bounded width into a path decomposition
of large order and bounded adhesion, and then to clean it up, making it linked, appearance-universal, 
and with large interiors.
The techniques to do so are standard and appear e.g. in~\cite{kntw}; we give brief arguments here to adjust for minor technical details and
notational differences.

\begin{observation}
A coarsening of a proper decomposition is proper.  A coarsening of a linked decomposition is linked.  A coarsening of
a decomposition of adhesion $p$ has adhesion at most $p$.
A coarsening of an appearance-universal decomposition is appearance-universal; furthermore, if an appearance-universal decomposition
has large interiors, then its coarsening has large interiors.
\end{observation}

\begin{lemma}\label{lemma-tree-to-path}
Let $k,n\ge 1$ be integers.  If a graph $G$ has a proper tree decomposition $(T,\beta)$ of order at least $n^n$ with bags of size at most $k$, then
$G$ has a proper path decomposition of adhesion at most $k$ and order at least $n$.
\end{lemma}
\begin{proof}
If $T$ has a vertex $z$ of degree $m\ge n$, then let $T_1$, \ldots, $T_m$ be the components of $T-z$, let $H$ be a path $z_1z_2\ldots z_m$
and let $\gamma(z_i)=\beta(z)\cup \bigcup_{x\in V(T_i)} \beta(x)$.  Otherwise, $T$ contains a subpath $H$ with at least $n$ vertices; for $z\in V(H)$,
let $T_z$ be the component of $T-(V(H)\setminus\{z\})$ containing $z$, and let $\gamma(z)=\bigcup_{x\in V(T_z)} \beta(x)$.
In both cases, $(H,\gamma)$ is a proper path decomposition of $G$ of adhesion at most $k$.
\end{proof}

\begin{lemma}\label{lemma-linked}
There exists a function $f_{\text{link}}:Z^+_0\times Z^+\to Z^+$ as follows.
Let $p$ and $n$ be integers.  If $G$ has a proper path decomposition $(H,\beta)$ with adhesion at most $p$ of order at least $f_{\text{link}}(p,n)$,
then $G$ has a proper linked path decomposition of adhesion at most $p$ and order $n$.
\end{lemma}
\begin{proof}
Choose $f_{\text{link}}$ so that $f_{\text{link}}(0,n)=n$ and $f_{\text{link}}(p,n)=f_{\text{link}}(p-1,n)+(f_{\text{link}}(p-1,n)-2)(n-1)-1$
for every $p,n\ge 1$.  We prove the claim by the induction on $p$.

If $p=0$, then $(H,\beta)$ is linked and the claim holds trivially.
If there exist at least $f_{\text{link}}(p-1,n)-1$ edges $z_1z_2\in E(H)$ such that $|\beta(z_1)\cap \beta(z_2)|\le p-1$, then $(H,\beta)$ has
a coarsening with adhesion at most $p-1$ and order at least $f_{\text{link}}(p-1,n)$, and the claim follows by the induction hypothesis.
Hence, assume that there are at most $f_{\text{link}}(p-1,n)-2$ such edges, and thus there exists a coarsening $(H_1,\beta_1)$ of $(H,\beta)$
of order $f_{\text{link}}(p,n)-f_{\text{link}}(p-1,n)+2$ such that any two adjacent bags intersect in exactly $p$ vertices.  If there exist at least
$f_{\text{link}}(p-1,n)-1$ vertices $z$ of $H_1$ with broken bags, then $G$ has a proper path decomposition with adhesion at most $p-1$ of order
at least $f_{\text{link}}(p-1,n)$, and the claim follows by induction.  Otherwise, since $(H_1,\beta_1)$ has order greater than $(f_{\text{link}}(p-1,n)-2)(n-1)$,
it contains $n$ consecutive vertices with unbroken bags, and thus it has a coarsening of order $n$ in that no bag is broken.
This coarsening is a proper linked path decomposition of $G$.
\end{proof}

\begin{lemma}\label{lemma-appun}
Let $p\ge 0$ and $n\ge 1$ be integers.  Every path decomposition $(H,\beta)$ of a graph $G$ with adhesion at most $p$ of order at least $n^{p+1}$
has an appearance-universal coarsening of order $n$.
\end{lemma}
\begin{proof}
We prove the claim by the induction on $p$.
If $p=0$, then every vertex appears in exactly one bag and $(H,\beta)$ is appearance-universal.
If a vertex $v$ appears in at least $n^p$ bags of the decomposition, then there exists a coarsening $(H_1,\beta_1)$
of order $n^p$ such that $v$ appears in all the bags.  Let $\beta'_1(z)=\beta(z)\setminus\{v\}$ for all $z\in V(H_1)$.
Then $(H_1,\beta'_1)$ is a path decomposition of $G-v$ of adhesion at most $p-1$, and by the induction hypothesis, it has a coarsening
$(H_2,\beta'_2)$ of order $n$ that is appearance-universal.  Let $\beta_2(z)=\beta(z)\cup\{v\}$ for all $z\in V(H_2)$.  Then $(H_2,\beta_2)$ is an
appearance-universal coarsening of $(H,\beta)$ of order $n$.

Hence, we can assume that every vertex appears in at most $n^p-1$ bags of the decomposition.
Let $(H',\beta')$ be the coarsening obtained by dividing $H$ into subpaths with $n^p$ vertices (plus possibly one shorter
path at the end) and merging the bags in the subpaths.  Then every vertex appears in at most two consecutive bags,
and thus $(H',\beta')$ is appearance-universal.
\end{proof}

\begin{lemma}\label{lemma-lint}
Let $n\ge 1$ be an integer.  Let $S\subseteq V(G)$.  Every proper appearance-universal path decomposition $(H,\beta)$ of order at least $3n$
has a coarsening of order $n$ with large interiors.
\end{lemma}
\begin{proof} Divide $H$ into subpaths with three vertices and merge the bags
in each subpath, obtaining a coarsening $(H_1,\beta_1)$ of order $n$.  We claim that this coarsening has large interiors.

Indeed, suppose that $z_0z_1z_2z_3z_4$ is a subpath of $H$ and consider a vertex $z\in V(H)$ such that $\beta_1(z)=\beta(z_1)\cup \beta(z_2)\cup\beta(z_3)$.
If a vertex $v$ belongs both to $\beta(z_2)$ and $\beta(z_0)$, then also $v\in\beta_1(z_1)$, and since $(H,\beta)$ is appearance-universal,
it follows that $v$ belongs to all its bags.  Similarly, any vertex in $\beta(z_2)\cap\beta(z_4)$ belongs to all bags.  Since $(H,\beta)$ is proper,
some vertex of $\beta(z_2)$ does not belong to all bags, and thus it is internal in $(H_1,\beta_1)$.

Let $z'$ and $z''$ be the neighbors of $z$ in $H_1$.  Consider now an internal vertex $v\in \beta_1(z)$.
Since $v$ is internal in $(H_1,\beta_1)$, it does not appear in all bags of $(H,\beta)$, and by appearance-universality, $v$ appears in at most two consecutive bags of $(H,\beta)$.
Consequently, $v$ does not appear in the bags of both $z_1$ and $z_3$, and by symmetry, we can assume that $v\not\in\beta(z_3)$.
Hence, any neighbor of $v$ in $\beta(z_4)$ must belong to $\beta(z_2)$, and by appearance-universality it must belong to all bags.  It follows that $v$ does not have a neighbor both in
$\beta_1(z')\setminus\beta_1(z'')$ and in $\beta_1(z'')\setminus\beta_1(z')$.
\end{proof}

A $4$-tuple $(F,M,l,r)$, where $F$ is a graph, $M$ is a union of $p$ pairwise vertex-disjoint paths in $F$, and $l,r:[p]\to V(F)$ are injective functions such that for $i=1,\ldots,p$,
$l(i)$ and $r(i)$ are the two endpoints of one of the paths of $M$, is called an \emph{extended bag of adhesion $p$}.
Let $(H,\beta)$ be a linked path decomposition of a graph $G$.  For each $z\in V(H)$ with two neighbors $x$ and $y$, fix a $(\beta(z)\cap\beta(x))-(\beta(z)\cap\beta(y))$ linkage $L_z$
in $G[\beta(z)]$.  Let $H_0$ be the path obtained from $H$ by removing its endpoints.
Let $L$ be the union of the linkages $L_z$ over all $z\in V(H_0)$.  Note that $L$ is the disjoint union of paths $L_1$, \ldots, $L_p$.
Order the path $H$ arbitrarily, and let $z$ be a vertex of $H_0$ whose predecessor in $H$ is $x$ and whose successor in $H$ is $y$.
Let $l_z:[p]\to\beta(z)$ and $r_z:[p]\to\beta(z)$ be functions defined so that $l_z(i)$ is the vertex of $L_i$ belonging to $\beta(z)\cap\beta(x)$ and $r_z(i)$ is the vertex of $L_i$ belonging to $\beta(z)\cap\beta(y)$.
The 4-tuple $E_z=(G[\beta(z)],L_z,l_z,r_z)$ is the \emph{extended bag of $z$}.  A \emph{finite extended bag property} is a function that to each extended bag of adhesion $p$ assigns an element of a finite set $X_p$.
By the pigeonhole principle, we have the following.

\begin{observation}\label{obs-prop}
For all integers $p\ge 0$ and $n\ge 1$ and for any finite extended bag property $\pi$, there exists an integer $N$ as follows.
For any linked decomposition $(H,\beta)$ of adhesion at most $p$ and any $Z \subseteq V(H)$ such that $|Z| \geq N$, there exists a set $U\subseteq Z$ of size $n$
such that $\pi(E_x)=\pi(E_y)$ for all $x,y\in U$.
\end{observation}

We are now ready to start working towards the proof of Theorem~\ref{thm-pbndtw}.  Let us start with a lemma that enables us to obtain
many $t$-islands separated from the rest of the graph by small cuts.

\begin{lemma}\label{lemma-main-tw}
For all integers $p\ge 0$ and $t,m,l\ge 1$, there exists an integer $N\ge 1$ as follows.
Let $(H,\beta)$ be a linked path decomposition of a graph $G$,
with large interiors, adhesion at most $p$ and order at least $N$.
Then either $G$ contains $K_{t,m}$ or $I_{t-1}+P_m$ as a minor, or there exists  a subpath $H'$ of $H$  of length $l$ such that the
internal vertices of the bag of $z$ form a $t$-island in $G$ for every $z \in V(H')$.
\end{lemma}
\begin{proof}
Suppose that there does not exist  a subpath $H'$ of $H$ as above.	
Let $Z$ be the set of vertices of $H$ of degree two such that the internal vertices of the bag of $z$ do not form a $t$-island in $G$. Then $|Z| \geq (N-2)/l$.
Consider any vertex $z\in Z$ of degree two, and let $(G[\beta(z)],L_z,l_z,r_z)$ be its extended bag.
Since the internal vertices of the bag of $z$ do not form a $t$-island, there exists an internal vertex $v_z\in \beta(z)$
with at least $t$ non-internal neighbors.  Choose $t$ of the non-internal neighbors $v_1$, \ldots, $v_t$, and let
$\pi(z)=(\sigma_1,\ldots,\sigma_t,\sigma_z)$, where for $i=1,\ldots,t$, we have $\sigma_i=j$ if $v_i=l_z(j)$ or $v_i=r_z(j)$,
and $\sigma_z=j$ if $z$ lies on the path in $L_z$ with ends $l_z(j)$ and $r_z(j)$, and $\sigma_z=0$ otherwise.
Observe that $\sigma_1$, \ldots, $\sigma_t$ are distinct, since $H$ has large interiors.
Note that $\pi$ is a finite extended bag property, and by Observation~\ref{obs-prop}, we can assume that there exist
distinct vertices $z_1$, \ldots, $z_m$ in $V(H)\setminus Z$ of degree two such that $\pi(z_1)=\pi(z_2)=\ldots=\pi(z_m)$.
Without loss of generality, we can assume that $\pi(z_1)=(1,2,\ldots,t,a)$, where $a\in \{0,1,t+1\}$.
Let $L$ be the union of the linkages $L_z$ over all vertices $z\in V(H)$ of degree two, and let $L_1$, \ldots, $L_p$ be the paths of $L$.

If $a\in \{0,t+1\}$, then contract each of the paths $L_1$, \ldots, $L_t$ to a single vertex $x_1$, \ldots, $x_t$; since $v_{z_1}$, \ldots, $v_{z_m}$
are adjacent to all of $x_1$, \ldots, $x_t$, we obtain $K_{t,m}$ as a minor of $G$.

If $a=1$, then contract each of the paths $L_2$, \ldots, $L_t$ to a single vertex $x_2$, \ldots, $x_t$,
and contract the parts of the path $L_1$ between vertices $v_{z_1}$, \ldots, $v_{z_m}$
so that they form a path on $m$ vertices.  We obtain $I_{t-1}+P_m$ as a minor of $G$.
\end{proof}

We utilize Lemma~\ref{lemma-main-tw} in an inductive argument as follows.

\begin{lemma}\label{lemma-comb}
For all integers $k\ge 0$ and $m,t\ge 1$, there exist real $\eps,C >0$  as follows.  Let $G$ be a graph with at least $C$
vertices, and let $S\subseteq V(G)$ be such that $|S| \leq  \eps|V(G)|+2k+3$.  If $G$ has tree-width at most $k$ and contains neither $K_{t,m}$ nor $I_{t-1}+P_m$ as a minor,
then $G$ contains a $t$-island disjoint from $S$.
\end{lemma}
\begin{proof}
Let $n_4$ be chosen so that Lemma~\ref{lemma-main-tw} holds for $p=k+1$, $t$, $m$, $l=4k+6$ and $N=n_4$.
Let $n_3=3n_4$, $n_2=n_3^{k+2}$, $n_1=f_{\text{link}}(k+1,n_2)$, and $C=(k+1)n_1^{n_1}$, and $\eps=\frac{1}{(2k+3)C}$. 
 
We prove that the lemma holds for these values of $\eps$ and $C$ by induction on the number of vertices of $G$.
Suppose that $G$ has a separation $(G_1,G_2)$ of order at most $(k+1)$ such that $|V(G_i)| \geq C$ for $i=1,2$.
Let $S_i = (S  \cap V(G_i)) \cup (V(G_1) \cap V(G_2))$ for $i=1,2$. If $|S_i| \leq \epsilon|V(G_i)|+2k+3$ for
some $i \in \{1,2\}$ then the lemma holds by the induction hypothesis applied to $G_i$. Otherwise, we have
$$|S| \geq |S_1|+|S_2| - 2(k+1) \geq \eps|V(G)|+2(2k+3)-2(k+1) >  \eps|V(G)|+2k+3,$$
a contradiction.
Thus we assume that no separation as above exists.

The graph $G$ has a proper tree decomposition $(T,\beta)$ with bags of size at most $k+1$.  Clearly, the decomposition has
order at least $|V(G)|/(k+1)\ge C/(k+1)=n_1^{n_1}$.  Hence, by Lemma~\ref{lemma-tree-to-path}, $G$ has a proper path decomposition of adhesion
at most $k+1$ and order $n_1$.  By Lemma~\ref{lemma-linked}, $G$ has a proper linked path decomposition of adhesion at most $k+1$
and order $n_2$.  By Lemma~\ref{lemma-appun}, this decomposition has an appearance-universal coarsening of order $n_3$,
and by Lemma~\ref{lemma-lint}, this decomposition can be further coarsened to a decomposition $(H,\beta)$ of order $n_4$ with large interiors.
By Lemma~\ref{lemma-main-tw}, there exists a subpath $H'$ of $H$  of length $l$ such that the set internal vertices $I_z$ of the bag of $z$ form a $t$-island in $G$ for every $z \in V(H')$.  If 
$I_z \cap S = \emptyset$ for some $z \in V(H')$ then $I_z$ is the required $t$-island, and so we assume  $I_z \cap S \neq \emptyset$ for every such $z$.
In particular, $|S|\ge 4k+6$, and thus $|V(G)|\ge (|S|-2k-3)/\eps\ge (2k+3)/\eps\ge(2k+3)^2C$.

Note that by our earlier assumption on the existence of certain separations,  we have that $|\beta(z)| \geq C$ for at most one $z \in V(H')$, and therefore we can select
a subpath  $H''$ of $H'$  of length $l/2=2k+3$ 
such that  $|\beta(z)| < C$ for every $z \in V(H'')$. Let $X=\bigcup_{z \in V(H'')}\beta(z)$ and
$Y= \bigcup_{z \in V(H)-V(H'')}\beta(z)$. Then $|X \cap Y| \leq 2k+2$, $|X| \leq (2k+3)C$, and $|S \cap (X \setminus Y)| \geq 2k+3$ by our assumptions.
Let $G'=G[Y]$ and $S'= (S \cap Y) \cup (X \cap Y)$. We have
\begin{align*}
|S'| &\leq |S|-1 =|S|-\eps(2k+3)C \le |S|-\eps|X|\\
&\leq \eps|V(G)|+2k+3-\eps|X| \leq \eps|V(G')|+2k+3.
\end{align*}
Since $|V(G)|\ge (2k+3)^2C$, we have $|V(G')|\ge |V(G)|-|X|\ge (2k+3)(2k+2)C>C$.
By the induction hypothesis, $G'$ contains a $t$-island $I$ disjoint from $S'$. Clearly, $I$ is a $t$-island in $G$ disjoint from $S$, as desired.
\end{proof}

The main result now readily follows.

\begin{proof}[Proof of Theorem~\ref{thm-pbndtw}]
By Theorem~\ref{thm-percol1} we have $\text{pcol}_\star(\GG)=p(\GG)$.  
By Observation~\ref{obs-basgr} we have $\text{pcol}_\star(\GG) \geq \text{col}_\star(\GG) \geq t$.
Finally, by Lemma~\ref{lemma-comb} we have $p(\GG) \leq t$ (if $G\in \GG$ has less than $C$ vertices,
then $G$ is $1/C$-resistant to $t$-percolation for every $t\ge 1$).
\end{proof}

\section{Concluding remarks}

In this paper we studied improper colorings of minor-closed graph classes, where we used the size of the maximum monochromatic component as a measure of impropriety. One can attempt building a similar theory for other impropriety measures as follows.  We say that a graph parameter $f$ with values in $\bb{R}_+ \cup \{+\infty\}$ is \emph{connected} if for a disconnected graph $G$ with components $G_1,G_2,\ldots,G_k$ we have $f(G)=\max_{1 \leq i \leq k}f(G_i)$. 

Let $f$ be a connected monotone graph parameter. 
  The \emph{$f$-chromatic number} $\chi_{f}(\GG)$  of a graph class $G$ as the minimum $t$ such that there exists real $C$ satisfying the following. For every $G  \in \mc{G}$ there exists a partition $V_1,\ldots,V_t$ of $V(G)$ so that $f(G[V_i]) \leq C$ for every $1 \leq i \leq t$. For example, Theorem~\ref{thm-devos} implies that $\chi_{\brm{tw}}(\GG) \leq 2$ for every minor-closed class $\GG$. 

One can also define the \emph{$f$-list chromatic number} $\chi^l_{f}(\GG)$  and the \emph{$f$-coloring number} $\col_{f}(\GG)$  of a graph class $G$ analogously to the clustered list chromatic and coloring numbers. The inequalities $\col_f(\mc{G})  \geq \chi^l_f(\mc{G}) \geq \chi_f(\mc{G})$ hold for any choice of $f$. 

Note that considering $f$ to be a connected graph parameter defined by $f(G)=1$, if $G$ is edgeless, and $f= +\infty$, otherwise,  one recovers from the definitions above the ordinary chromatic, list-chromatic numbers, and the coloring number.

The graph parameter $\Delta$ (maximum degree) is particularly well-studied in this context.  Edwards et al.~\cite{EKKOS15} have shown that $\chi_\Delta(\mc{X}(K_t))=t-1$. Ossona de Mendez, Oum and Wood~\cite{OOW16} characterized minor closed graph classes with given $\Delta$-list chromatic number as follows.

\begin{theorem}\label{thm-Deltalist}
	A minor-closed class of graphs $\GG$ satisfies $\chi^l_\Delta(\GG)\le t$ if and only if there exists $m\ge 1$ such that
	$K_{t,m}\not\in \GG$.
\end{theorem}

As we mentioned before, Conjecture~\ref{conj-chara} implies that $\col_\star(\GG)= \chi^l_\star(\GG)$ for every minor-closed class $\GG$.
Analogously, Theorem~\ref{thm-Deltalist} and Conjecture~\ref{conj-chara} imply that
$$\chi^l_\Delta(\GG) \leq \col_\Delta(\GG)\leq \col_\star(\GG) \leq \chi^l_\Delta(\GG) +1.$$ 
Unfortunately, the left inequality does not always holds with equality, as shown in the next lemma.

\begin{lemma}
Let $\mc{O}$ be the class of outerplanar graphs. Then $\chi^l_\Delta(\mc{O})=2$ and $ \col_\Delta(\mc{O})=3$.
\end{lemma} 
\begin{proof}
We have $\chi^l_\Delta(\mc{O})=2$ by Theorem~\ref{thm-Deltalist}. The treewidth of graphs in $\mc{O}$ is at most $2$ and $K_{2,3} \not \in \mc{O}$. Thus it follows from Theorem~\ref{thm-bndtw} that $\col_\star(\mc{O}) \leq 3$, and therefore $ \col_\Delta(\mc{O}) \leq 3$.

It remains to show that $\col_\star(\mc{O}) \geq 3$ , i.e. for every $C>0$ there exists an outerplanar graph $G$ such that $G$ contains no $2$-island $S$ such that $\Delta(G[S]) \leq C$. Indeed,  consider an outerplanar graph constructed as follows. 
Let $P_1,P_2$ and $P_3$ be three vertex-disjoint paths on $C+1$ vertices. Let $v_i$ be an end of $P_i$ for $i=1,2,3$. The graph $G$ is obtained from $P_1 \cup P_2 \cup P_3$ by joining $v_i$ by an edge to every vertex of $P_{i+1}$ for $i=1,2,3$, where $P_4=P_1$ by convention.
Let $S$ be a $2$-island in $G$. Suppose that $v_1,v_2,v_3 \not \in S$. Then without loss of generality there exists a vertex in $v \in V(P_2) \cap S$. Choose such a vertex $v$ closest to $v_2$ along $P_2$. Then $v$ has two neighbors not in $S$: one along $P_2$ and $v_1$. This contradiction implies that $v_i \in S$ for some $i \in \{1,2,3\}$. As $v_i$ has degree $C+2$, at least $C+1$ of its neighbors lie in $S$, implying $\Delta(G[S]) > C$, as desired.
\end{proof}

Following this line of inquiry one might ask how closely $\col_f(\GG)$ and $\chi^l_f(\GG)$ are related for other
(natural) connected graph parameters $f$.  Dvo\v{r}\'ak, Pek\'arek, and Sereni~\cite{colandch} study this question
in more detail.

\bibliographystyle{siam}
\bibliography{smallcol}

\end{document}